\documentclass[11pt]{elsarticle}
\usepackage{amsmath,mathrsfs}
\usepackage{amsfonts}
\usepackage{amsmath}
\usepackage{amssymb}
\usepackage{graphicx,subfigure,color}
\usepackage{algorithm}
\usepackage{multirow}
\usepackage{xcolor}
\usepackage{setspace}
\usepackage{mathdots}
\usepackage{subfigure}
\usepackage{epstopdf}
\usepackage{booktabs}
\usepackage[title]{appendix}
\usepackage{hyperref}
\usepackage{cleveref}
\newtheorem{proposition}{Proposition}
\biboptions{numbers,sort,compress}
\usepackage{latexsym,array,graphicx,extarrows}
\providecommand{\U}[1]{\protect\rule{.1in}{.1in}}
\allowdisplaybreaks[4]

\usepackage{algpseudocode,float}



\numberwithin{equation}{section}
\newenvironment{AMS}{\noindent{\bf Mathematics Subject Classification:}}

\newtheorem{remark}{Remark}
\newtheorem{theorem}{Theorem}
\newtheorem{example}{Example}
\newtheorem{lemma}[theorem]{Lemma}

\newtheorem{definition}{Definition}

\newenvironment{proof}{\noindent{\bf Proof:}}{\hfill\fbox{}\vspace*{1mm}}

\newcommand{\cred}[1]{{\color{black} #1}}

\begin{document}	
	
	\begin{frontmatter}
		\title{A block $\alpha$-circulant based preconditioned MINRES method for wave equations}
		
		\author{Xue-lei Lin\footnote{School of Science, Harbin Institute of Technology,
				Shenzhen 518055, China. E-mail: linxuelei@hit.edu.cn} \textrm{and}
			Sean Hon\footnote{The Corresponding Author. Department of Mathematics, Hong Kong Baptist University, Kowloon Tong, Hong Kong SAR.
				E-mail: seanyshon@hkbu.edu.hk}}
		
		\begin{abstract}
			In this work, we propose an absolute value block $\alpha$-circulant preconditioner for the minimal residual (MINRES) method to solve an all-at-once system arising from the discretization of wave equations. Motivated by the absolute value block circulant preconditioner proposed in [E. McDonald, J. Pestana, and A. Wathen. SIAM J. Sci. Comput., 40(2):A1012–A1033, 2018], we propose an absolute value version of the block $\alpha$-circulant preconditioner. Since the original block $\alpha$-circulant preconditioner is non-Hermitian in general, it cannot be directly used as a preconditioner for MINRES. Our proposed preconditioner is the first Hermitian positive definite variant of the block $\alpha$-circulant preconditioner for the concerned wave equations, which fills the gap between block $\alpha$-circulant preconditioning and the field of preconditioned MINRES solver. The matrix-vector multiplication of the preconditioner can be fast implemented via fast Fourier transforms. Theoretically, we show that for a properly chosen $\alpha$ the MINRES solver with the proposed preconditioner achieves a linear convergence rate independent of the matrix size. To the best of our knowledge, this is the first attempt to generalize the original absolute value block circulant preconditioner in the aspects of both theory and performance the concerned problem. Numerical experiments are given to support the effectiveness of our preconditioner, showing that the expected optimal convergence can be achieved.
		\end{abstract}
		
		\begin{keyword}
			Block Toeplitz matrix; convergence of MINRES solver; wave equations; absolute value block $\alpha$-circulant preconditioners
			
			\begin{AMS}
				15B05; 65F08; 65F10;  65M22
			\end{AMS}
		\end{keyword}
		
	\end{frontmatter}
	
\section{Introduction}\label{sec:intro}

Over the past few decades, the development of parallel-in-time (PinT) solvers for evolutionary partial differential equations (PDEs) has been an active research area. Various PinT methods have been proposed and, among these, two types of solvers are particularly of interest due to their success for a wide range of equations - the parareal method \cite{LIONS2001661} and the multigrid-reduction-in-time \cite{doi:10.1137/130944230}. For more on the development of such parallel methods for PDEs, we refer readers to {\cite{Bertaccini_Ng_2003,Bertaccini_2000,Bertaccini_Durastante_2018,doi:10.1137/0916050,10.1007/978-3-319-18827-0_50,doi:10.1137/05064607X,doi:10.1137/15M1046605,10.1007/978-3-319-23321-5_3,doi:10.1137/16M1074096}} and the references therein.

Recently, ever since the inception of the preconditioning method developed in \cite{doi:10.1137/16M1062016}, there have been an increasing amount of work on developing effective PinT preconditioners for the all-at-once systems arising from solving evolutionary PDEs (see, e.g., \cite{https://doi.org/10.1002/nla.2386,0x003ab34e,LIN2021110221,HonFungDongSC_2023,hon_SCapizzano_2023}). These methods can be categorized into the so-called all-at-once PDE solvers. In other words, instead of solving the target PDE in a sequential way, they propose a PinT preconditioner for solving the vast all-at-once linear system, which is constructed from stacking all unknowns at each time levels simultaneously. Among these methods, block $\alpha$-circulant preconditioners have attracted much attention due to their superior performance for a large class of PDEs, including heat equations \cite{doi:10.1137/20M1316354}, wave equations \cite{doi:10.1137/19M1309869,s10915-021-01701-x}, and advection-diffusion equations and viscous Burgers' equation \cite{WU2021110076}. However, block $\alpha$-circulant preconditioners are nonsymmetric and, therefore, not Hermitian positive definite (HPD). Consequently, they cannot be directly utilized with symmetric preconditioned Krylov subspace solvers such as the minimal residual (MINRES) method. One of the objectives of this work is to bridge this gap, thereby enabling block $\alpha$-circulant preconditioners to accelerate the convergence rate of the MINRES solver when solving the all-at-once system.

Letting alone the preconditioner used, the MINRES method is only applicable to solving symmetric linear system. Thanks to the symmetrization technique proposed in \cite{McDonald2017}, the all-at-once system can be equivalently converted into a symmetric linear system by applying a reverse-time ordering permutation transformation. In \cite{doi:10.1137/16M1062016}, an absolute value block circulant (ABC) preconditioner was proposed for the symmetrized linear system; this preconditioner is HPD and, therefore, suitable for use with the MINRES solver. The numerical results presented in \cite{doi:10.1137/16M1062016} demonstrate that the MINRES solver, when coupled with the ABC preconditioner, is effective for the symmetrized all-at-once system that arises from parabolic equations. Nevertheless, as will be shown from the numerical examples in Section \ref{sec:numerical} and the explanations in Remarks \ref{p1badpermformremark} \& \ref{c1badperformremark}, MINRES with the ABC preconditioner does not work well for wave equations. Hence, another objective of this work is to improve the performance of the ABC preconditioner.

Inspired by the ABC preconditioning technique, we consider proposing an absolute-value version of the aforementioned block $\alpha$-circulant preconditioner, which is HPD and thus can be used as a preconditioner of MINRES solver. In fact, our proposed absolute value block $\alpha$-circulant (ABAC) preconditioner can be seen as an extension of the ABC preconditioner. In other words, when $\alpha=1$, the ABAC preconditioner reduces to the ABC preconditioner. Like the original block $\alpha$-circulant preconditioner, the implementation of the proposed ABAC preconditioner is PinT.	Theoretically, we show that the MINRES solver with the proposed ABAC preconditioner has a matrix-size independent convergence rate for solving the symmetrized all-at-once linear system from wave equation when $\alpha$ is properly chosen. Numerical results in Section \ref{sec:numerical} also show that the ABAC preconditioner outperforms the ABC preconditioner.

More importantly, to the best of our knowledge, our proposed preconditioning method is the first work to show that a nontrivial symmetric positive definite version of block $\alpha$-circulant matrix can be constructed for the symmetrized the all-at-once system arising from wave equations, providing a positive answer to the open question discussed in \cite[Section~3]{doi:10.1137/19M1309869} that the construction and analysis of such a preconditioner is possible. We emphasize that the construction of our proposed preconditioner, involving the use of matrix square root on $\alpha$-circulant matrices, has rarely been seen in the relevant literature. Through this novel preconditioning idea, we hope to inspire and stimulate a new line of efficient algorithms for solving other more challenging PDEs, further enriching the PinT solver community. 

It is worth noting that the use of ABAC preconditioners has recently been applied to dense block lower triangular Toeplitz systems \cite{LiLinHonWU_2023}. However, we must emphasize that the theoretical results presented in \cite{LiLinHonWU_2023} are not applicable to the wave equation considered in this work. This inapplicability arises because \cite[Assumption 1 (ii)]{LiLinHonWU_2023} is not satisfied by the matrix $\mathcal{T}$ to be defined in \eqref{eq:main_matrix}.

In this work, we are interested in solving the following linear wave equation as a model problem
\begin{equation}\label{eqn:wave}
	\left\{
	\begin{array}{lc}
		u_{tt}(x,t) = \Delta u(x,t)+ f(x,t), & (x,t)\in \Omega \times (0,T], \\
		u = 0, & (x,t)\in \partial \Omega \times (0,T], \\
		u(x,0)=\psi_0,~u_t(x,0)=\psi_1, & x \in \Omega.
	\end{array}
	\right.\,
\end{equation} {For instance, a direct PinT solver was proposed in \cite{Gander_Halpern_etc_2019} for (\ref{eqn:wave}), where the so-called diagonalization technique \cite{Maday_Eonquist_2008} was used. For more about the existing solvers based on this diagonalization technique for solving the wave equation, we refer to the recent review paper \cite{2020arXiv200509158G}.} 

{Unlike a direct solver, we focus on a preconditioned iterative method for the equation in what follows.} For solving (\ref{eqn:wave}), we adopt here the implicit leap-frog finite difference scheme used in \cred{\cite{doi:10.1137/19M1309869,doi:10.1137/19M1289613}}, which was established {and shown in \cite{LiLiuXiao2015} to be unconditionally stable without the requirement of the restrictive Courant-Friedrichs-Lewy (CFL) condition on the mesh step sizes}. \textcolor{black}{Specifically, given a positive integer $n$, we let $\tau=\frac{T}{n}$ be the time step size. Let $h$ denote the step size along each spatial direction and let $m$ denote the number of spatial grid points.} For $k = 1,\dots, n-1$, we have
\[ \frac{\mathbf{u}_m^{(k+1)}-2\mathbf{u}_m^{(k)} + \mathbf{u}_m^{(k-1)} }{\tau^2} =  \Delta_h    \frac{\mathbf{u}_m^{(k+1)} +  \mathbf{u}_m^{(k-1)}}{2}  +  \mathbf{f}^{(k)} ,
\]
where $\Delta_h \in \mathbb{R}^{m\times m}$ is the second-order discrete matrix approximating the Laplacian operator $\Delta$ in (\ref{eqn:wave}), $\mathbf{u}_m^{(k)}\in\mathbb{R}^{m\times 1}$ is an approximation to $u(\mathcal{G}_h,k\tau)$,  $\mathbf{f}^{(k)}=f(\mathcal{G}_h,k\tau)$,  $\mathcal{G}_h$ denoting the spatial grid points arranged in the lexicographic ordering.

Instead of solving the above equations for $\mathbf{u}_m^{(k)}$ sequentially for $k=1,2,\ldots, n-1$, we consider the following equivalent all-at-once system $mn$-by-$mn$ (real) nonsymmetric block Toeplitz system

\begin{equation}\label{eq:main_matrix}
	\underbrace{\begin{bmatrix}
			L &   &  &  & \\
			-2I_{m}  & L    & & & \\
			L &  -2I_{m}  & L  & &  \\
			&   \ddots & \ddots & \ddots &  \\
			&  &  L  & -2I_{m} & L
	\end{bmatrix}}_{=:\mathcal{T}}
	\underbrace{\left[\begin{array}{c}\mathbf{u}_m^{(1)}\\ \mathbf{u}_m^{(2)}\\ \mathbf{u}_m^{(3)}\\ \vdots \\ \mathbf{u}_m^{(n)}
		\end{array} \right]}_{=:\mathbf{u}}
	= \underbrace{\tau^2 \left[\begin{array}{c} \mathbf{f}_m^{(0)}/2+ \Psi_1/\tau + \Psi_0/\tau^2  \\ 
			\mathbf{f}_m^{(1)} - L \Psi_0 /\tau^2\\
			\mathbf{f}_m^{(2)} \\
			\vdots \\
			\mathbf{f}_m^{(n-1)}
		\end{array} \right]}_{=:\mathbf{f}},
\end{equation}where  $L = I_{m}  -  \frac{\tau^2}{2} \Delta_h \in \mathbb{R}^{m\times m}.$ Note that $I_{m}$ denotes the $m \times m$ identity matrix.

Throughout, the discrete negative Laplacian matrix $-\Delta_{h}$ is assumed HPD and hence $L$ is HPD as well. Such an assumption is easily satisfied when a finite difference method is used on a uniform grid. In a more general case where the spatial domain is irregular and a finite element method is used \cite{ElmanSilvesterWathen2004}, the identity matrix $I_{m}$ and $-\Delta_{h}$ are replaced by the mass matrix and stiffness matrix, respectively.

We remark that for solving the wave equation (\ref{eqn:wave}) one can apply other discretization schemes and thus obtain different linear systems other than (\ref{eq:main_matrix}). For example, \cite{https://doi.org/10.1002/nla.2386,0x003ab34e,hon_SCapizzano_2023} concern various time stepping methods while \cite{s10915-021-01701-x} adopts the Numerov method.

Thanks to the symmetrization technique proposed in \cite{McDonald2017}, we consider \cred{to} rewrite \eqref{eq:main_matrix} as the following equivalent symmetric form
\begin{equation}\label{symmetricform}
	\mathcal{Y} \mathcal{T} \mathbf{u} = \mathcal{Y}\mathbf{f},
\end{equation} 
where
\begin{equation}\label{eqn:main_system}
	\mathcal{Y} \mathcal{T}
	=
	\begin{bmatrix}
		&  &  L  & -2I_{m} & L\\
		&   \iddots & \iddots & \iddots &  \\
		L &  -2I_{m}  & L  & &  \\
		-2I_{m}  & L    & & & \\
		L &   &  &  & \\
	\end{bmatrix},
\end{equation}
and $\mathcal{Y} = Y_{n} \otimes I_{m}$ with $Y_{n} \in \mathbb{R}^{n \times n}$ being the anti-identity matrix, i.e., $[Y_n]_{j,k}=1$ if and only if $j+k=n+1$ and $[Y_n]_{j,k}=0$ elsewhere. Clearly, the monolithic matrix $\mathcal{Y} \mathcal{T}$ is symmetric since $L$ is assumed symmetric.

\cred{
For the symmetrized matrix $ \mathcal{Y} \mathcal{T}$ given in (\ref{eqn:main_system}), we know from {\cite[Theorem 3.4]{Ferrari2019}} and previous results in \cite{MazzaPestana2018,MR1671591,MazzaPestana2021} that its eigenvalues can be determined from the so-called generating function of $\mathcal{T}$. According to these results, $ \mathcal{Y} \mathcal{T}$ is symmetric indefinite when $n$ is large enough, which accounts for the use of MINRES instead of a HPD Krylov subspace method like the conjugate gradient method in this work. For complete surveys on generating functions as well as other properties of block multilevel Toeplitz matrices, we refer readers to \cite{MR2108963,Chan:1996:CGM:240441.240445,book-GLT-I,book-GLT-II,MR2376196} and the references therein. 
}

\cred{Proposed in \cite{doi:10.1137/19M1309869}, the following block $\alpha$-circulant preconditioner can be constructed as an approximation to $\mathcal{T}$:
\begin{eqnarray}\label{def:matrix_C_alpha}
	\mathcal{C}_\alpha ={\begin{bmatrix}
			L &   &  &  \alpha L & -2\alpha I_{m}\\
			-2I_{m}  & L    & & &  \alpha L \\
			L &  -2I_{m}  & L  & &  \\
			&   \ddots & \ddots & \ddots &  \\
			&  &  L  & -2I_{m} & L
	\end{bmatrix}}.
\end{eqnarray}
Clearly, $\mathcal{C}_{\alpha}$ is not symmetric and, thus, not HPD. However, Krylov subspace solvers (e.g., MINRES) for symmetric linear systems typically require an HPD preconditioner for maintaining the short-recurrence property. Therefore, $\mathcal{C}_{\alpha}$ is not a suitable preconditioner for the symmetric linear system \eqref{symmetricform}.}

{To address this issue of positive definiteness,  we first note that $\mathcal{C}_{\alpha} $ is diagonalizable for $\alpha\in(0,1)$. To see this, we can rewrite $\mathcal{C}_\alpha$ as
\begin{eqnarray}\label{eqn:C_decom}
	\mathcal{C}_\alpha &=& B_1^{(\alpha)} \otimes L + B_2^{(\alpha)} \otimes (-2I_{m}),
\end{eqnarray}
where 
\begin{equation*}\label{eqn:matricesB1B2}
	B_1^{(\alpha)} = {\begin{bmatrix}
			1 &   &  &  \alpha& 0 \\
			0  & 1    & & &  \alpha  \\
			1 &  0  & 1  & &  \\
			&   \ddots & \ddots & \ddots &  \\
			&  &  1  & 0 & 1
	\end{bmatrix}}, \quad B_2^{(\alpha)} = {\begin{bmatrix}
			0 &   &  &   & \alpha  \\
			1  & 0    & & &    \\
			&  1  &  0 & &  \\
			&    & \ddots & \ddots  &  \\
			&  &    & 1 & 0
	\end{bmatrix}}
	\in 
	\mathbb{R}^{n \times n}
\end{equation*}
According to \cite{BiniLatoucheMeini}, $B_1^{(\alpha)}$ and $B_2^{(\alpha)}$ are diagonalizable as follows
\begin{align}
&B_j^{(\alpha)} = D_{\alpha}^{-1}\mathbb{F}\Lambda_{\alpha,j}\mathbb{F}^{*}D_{\alpha},~j=1,2, \notag\\
&D_{\alpha}={\rm diag}(\alpha^{\frac{i-1}{n}})_{i=1}^{n},\quad \mathbb{F}=\frac{1}{\sqrt{n}}[\theta_n^{(i-1)(j-1)}]_{i,j=1}^{n},\quad\theta_n =\exp\left(\frac{2\pi {\bf i}}{n}\right),\quad{\bf i}=\sqrt{-1},\notag\\
&\Lambda_{\alpha,j}={\rm diag}(\lambda_{j,k}^{(\alpha)})_{k=1}^{n},\quad j=1,2,\quad \lambda_{1,k}^{(\alpha)}=1+\alpha^{\frac{2}{n}}\theta_n^{2(k-1)},\quad \lambda_{2,k}^{(\alpha)}=-2\alpha^{\frac{1}{n}}\theta_n^{k-1}.\notag
\end{align}
With the diagonalization formulas, $\mathcal{C}_{\alpha} $ ($\alpha\in(0,1)$) can be shown block diagonalizable as follows 
\begin{align}
\mathcal{C}_{\alpha}&=[(D_{\alpha}^{-1}\mathbb{F})\otimes I_m](\Lambda_{\alpha,1}\otimes L+\Lambda_{\alpha,2}\otimes I_m)[(\mathbb{F}^{*}D_{\alpha})\otimes I_m].\label{calphaprediagform}
\end{align}
Since $L$ is HPD, $\Lambda_{\alpha,1}\otimes L+\Lambda_{\alpha,2}\otimes I_m$ is clearly diagonalizable. In other words, $\mathcal{C}_{\alpha}$ is similar to a diagonalizable matrix and thus $\mathcal{C}_{\alpha}$ itself is  diagonalizable for $\alpha\in(0,1)$. That means for $\alpha\in(0,1)$ there exists an invertible matrix ${\bf X}$ and a diagonal matrix ${\rm diag}(\mu_i)_{i=1}^{nm}$ such that
\begin{equation*}
\mathcal{C}_{\alpha}={\bf X}^{-1}{\rm diag}(\mu_i)_{i=1}^{nm}{\bf X}.
\end{equation*}
Thus, we can define a matrix square root of $\mathcal{C}_{\alpha}$ for $\alpha\in(0,1)$ as follows
\begin{equation}\label{calphasqrtdiagform}
\mathcal{C}_{\alpha}^{1/2}={\bf X}^{-1}{\rm diag}(\sqrt{\mu_i})_{i=1}^{nm}{\bf X},
\end{equation}
where $\sqrt{\mu_i}$ denotes the principal branch of complex square root of  $\mu_i$.
}

Consequently, we in this work propose the following HPD preconditioner for $\mathcal{Y} \mathcal{T}$: 
\begin{equation}\label{def:matrix_P}
	\mathcal{P}_{\alpha}={({\mathcal{C}_\alpha^{1/2}})^{T}} {\mathcal{C}}_\alpha^{1/2}.
\end{equation}

{
	Then, instead of directly solving $\mathcal{YT}{\bf u}=\mathcal{Y}{\bf f}$, we employ the MINRES solver to solve the following equivalent preconditioned system
	\begin{equation}\label{precsys}
		\mathcal{P}_{\alpha}^{-1}\mathcal{YT}{\bf u}=\mathcal{P}_{\alpha}^{-1}\mathcal{Y}{\bf f}.
	\end{equation}
}

For a real symmetric matrix ${\bf H}$, denote by $|{\bf H}|$, the positive semi-definite matrix obtained by replacing eigenvalues of ${\bf H}$ with its absolute value.

\begin{remark}\label{p1badpermformremark}
	When $\alpha=1$, $\mathcal{P}_{\alpha}={({\mathcal{C}_1^{1/2}})^{T}} \mathcal{C}_1^{1/2}=\sqrt{\mathcal{C}^T_1 \mathcal{C}_1}=|\mathcal{C}_1|$ reduces to the existing ABC preconditioner proposed in \cite{doi:10.1137/16M1062016} for heat equations. {However, as demonstrated by the numerical results in Section \ref{sec:numerical}, $\mathcal{P}_{1}$ does not perform well for wave equations. \cred{The reason is that the eigenvalues of preconditioned matrix by the ABC preconditioner have outliers away from $\{-1,+1\}$ and the number of outliers is of $\mathcal{O}(m)$, which means the iteration number needed for convergence increases as $m$ increases.} This is the motivation for proposing $\mathcal{P}_{\alpha}$ in this work. As will be shown, Theorem \ref{maincvgthm} gives choices of $\alpha\in(0,1)$ so that the MINRES solver for \eqref{precsys} has a linear convergence rate independent of $h$ and $\tau$. With guidance of Theorem \ref{maincvgthm}, results in Section \ref{sec:numerical} show that $\mathcal{P}_{\alpha}$ outperforms  $\mathcal{P}_{1}$ for wave equations.}
\end{remark}

Ultimately, we will show that the preconditioned MINRES solver for the preconditioned system \eqref{precsys} achieves a convergence rate independent of $\tau$ and $h$ and therefore independent of the matrix size. 

This paper is organized as follows. We present our main results in Section \ref{sec:main} in which the convergence of our proposed preconditioner is analyzed.  Numerical tests are given in Section \ref{sec:numerical} to demonstrate the matrix-size independent convergence rate of the proposed preconditioned MINRES solver.

\section{Main results}\label{sec:main}

In this section, we discuss properties of our proposed ABAC preconditioner $\mathcal{P}_{\alpha}$ defined by (\ref{def:matrix_P}) \cred{and the fast implementation for the proposed preconditioned method}.

\cred{

\subsection{Preliminaries for convergence analysis of the preconditioned method}\label{subsec:analysis}
In this subsection, we introduce preliminaries for the linear convergence rate analysis of MINRES solver for the preconditioned system.

Denote by $GL(K)$ the set of all $K\times K$ invertible complex matrices. For a square matrix ${\bf A}$, denote by $\sigma({\bf A})$, the spectrum of ${\bf A}$. Define
\begin{equation*}
\mathcal{Q}(k):=\{{\bf A}\in\mathbb{C}^{K\times K}|\sigma({\bf A})\subset\mathbb{C}\setminus(-\infty,0]\}.
\end{equation*}
Clearly, $\mathcal{Q}(K)\subset GL(K)$. For a complex number $z$, denote by $\Re(z)$ and $\Im(z)$, the real part and the imaginary part of $z$, respectively.

\begin{lemma}\label{calphanonsinglm}
\begin{description}
For any $\alpha\in[0,1)$, it holds  $\mathcal{C}_{\alpha}\in\mathcal{Q}(mn)\subset GL(mn)$.
\end{description}
\end{lemma}
\begin{proof}
	See the proof in Appendix ${\bf A}$.
\end{proof}

}

\begin{remark}\label{c1badperformremark}
	Clearly, Lemma \ref{calphanonsinglm} guarantees that the preconditioner $\mathcal{P}_{\alpha}$ is invertible. Moreover, from its proof, we see that the eigenvalues of {Strang's \cite{Strang1986}} block circulant preconditioner $\mathcal{C}_1$ are of form $\mu=2\cos(\omega_k)[\lambda\cos(\omega_k)-1]+{\bf i}2\sin(\omega_k)[\lambda\cos(\omega_k)-1]$ and some of which are equal to zero when $\lambda\cos(\omega_k)-1=0$. Therefore, such an existing block circulant preconditioner may be singular and thus not applicable in general.
\end{remark}

\cred{
Denote by ${\rm UTT}({\bf v})$, the upper triangular Toeplitz matrix with ${\bf v}$ as its first column.
Define the $i$-th order Jordan block $J_i(\lambda)$ as
\begin{equation*}
	J_i(\lambda):=\begin{cases}
		\lambda,\quad i=1,\\
		\left[
		\begin{array}[c]{cccc}
			\lambda&1&&\\
			&\lambda&\ddots&\\
			&&\lambda&1\\
			&&&\lambda
		\end{array}
		\right]\in\mathbb{C}^{i\times i},\quad i>1.
	\end{cases}
\end{equation*}
\begin{definition}\textnormal{(Matrix function via Jordan canonical form \cite[Definition 1.2]{Higham2008})}\label{matfuncdef}
	Let ${\bf A}={\bf Z}{J}{\bf Z}^{-1}$ be a Jordan canonical form of a matrix ${\bf A}\in\mathbb{C}^{N\times N}$ with
	\begin{equation*}
		J={\rm blockdiag}(J_{n_i}(\lambda_i))_{i=1}^{k},\quad \{\lambda_i\}_{i=1}^{k}\in\sigma({\bf A}),\quad \sum\limits_{i=1}^{k}n_i=N.
	\end{equation*}
	For a function $g$ of one complex variable that is sufficiently smooth on each point of $\sigma({\bf A})$, a matrix function is defined by
	\begin{equation*}
		g({\bf A}):={\bf Z}{\rm blockdiag}(g(J_{n_i}(\lambda_i)))_{i=1}^{k}{\bf Z}^{-1},
	\end{equation*} 
	where $g(J_{n_i}(\lambda_i))$ is an upper triangular Toeplitz matrix given as
	\begin{equation*}
		g(J_{n_i}(\lambda_i))={\rm UTT}\left(\left(\frac{g(\lambda_i)}{0!},\frac{g^{\prime}(\lambda_i)}{1!},...,\frac{g^{(j)}(\lambda_i)}{j!},...,\frac{g^{(n_i-1)}(\lambda_i)}{(n_i-1)!}\right)\right)\in\mathbb{C}^{n_i\times n_i}.
	\end{equation*}
\end{definition}

For $z\in\mathbb{C}$, define $s(z)$ as the principal branch of complex square root of $z$.
\begin{proposition}\textnormal{(\cite[Example 1.16]{petrini2018guide})}\label{psqsmoothprop}
	The principal square root function $s(\cdot)$ is holomorphic on $\mathbb{C}\setminus(-\infty,0]$.
\end{proposition}

By Definition \ref{matfuncdef} and Proposition \ref{psqsmoothprop}, $s(\cdot)$ as a matrix function is well-defined on $\mathcal{Q}(K)$ for any positive integer $K$.

Denote $\mathbb{C}_+:=\{z\in\mathbb{C}|\Re(z)>0\}$.

\begin{lemma}\textnormal{(\cite[Theorem 1.29]{Higham2008})}\label{psqrtuniqlm}
For ${\bf A}\in \mathcal{Q}(K)$ with any positive integer $K$, 
\begin{description}
\item [(a)] there exists a unique matrix ${\bf B}$ such that ${\bf B}^2={\bf A}$ and $\sigma({\bf B})\subset\mathbb{C}_{+}$;
\item[(b)] the matrix ${\bf B}$ is exactly $s({\bf A})$;
\item[(c)] if, additionally, ${\bf A}$ is real, then $s({\bf A})$ is real.
\end{description}
\end{lemma}
\begin{proof}
The results of Lemma \ref{psqrtuniqlm} come from \cite[Theorem 1.24]{Higham2008} and  \cite[Theorem 1.29]{Higham2008}.
\end{proof}

Since $\sigma(\mathcal{T})=\sigma(L)\subset(1,+\infty)\subset\mathbb{C}\setminus(-\infty,0]$, $\mathcal{T}\in\mathcal{Q}(mn)$. That means $s(\mathcal{T})$ is well defined. In the rest of this paper, we use the notation $\mathcal{T}^{\frac{1}{2}}$ to represent  $s(\mathcal{T})$.

}

\begin{lemma}\label{lemma:real_T}
	Let $\mathcal{T} \in \mathbb{R}^{mn \times mn} $ be defined by (\ref{eq:main_matrix}). Then, 
	\begin{description}
		\item[(a)] ${\mathcal{T}}^{1/2}$ is invertible;
		\item[(b)] ${\mathcal{T}}^{1/2}$ is a real-valued matrix;
		\item[(c)] $\mathcal{Y} {\mathcal{T}}^{1/2}$ is symmetric.
	\end{description}
\end{lemma}

\begin{proof}
\cred{
As discussed above, $\sigma(\mathcal{T})\subset(1,+\infty)$. Therefore,  $\sigma({\mathcal{T}}^{1/2})\subset(1,+\infty)$, which proves part (a).

Since $\mathcal{T}$ is real-valued, Part (b) of Lemma \ref{lemma:real_T} follows from Part (c) of Lemma \ref{psqrtuniqlm}.
}

We now show part $(c)$ of Lemma \ref{lemma:real_T}. Since $\mathcal{Y} {\mathcal{T}}$ is symmetric by (\ref{eqn:main_system}) (i.e., $\mathcal{Y} {\mathcal{T}} \mathcal{Y}^{-1} = {\mathcal{T}}^T$) and $(\mathcal{Y} {\mathcal{T}} \mathcal{Y}^{-1})^{1/2} = \mathcal{Y} {\mathcal{T}}^{1/2} \mathcal{Y}^{-1}$ by \cite[Theorem 1.13 (c)]{Higham2008}, we can see that $(\mathcal{Y} {\mathcal{T}} \mathcal{Y}^{-1})^{1/2} = \mathcal{Y} {\mathcal{T}}^{1/2} \mathcal{Y}^{-1}=({\mathcal{T}}^T)^{1/2} = {({\mathcal{T}^{1/2}})^{T}}$, where the last equation follows by \cite[Theorem 1.13 (b)]{Higham2008}. Thus, we readily have shown that $\mathcal{Y} {\mathcal{T}}^{1/2} = {({\mathcal{T}^{1/2}})^{T}} \mathcal{Y}=(\mathcal{Y} {\mathcal{T}}^{1/2})^T$. Part $(c)$ is concluded.
\end{proof}

The following proposition indicates that an ideal preconditioner for $\mathcal{Y}\mathcal{T}$ is the HPD matrix ${({\mathcal{T}^{1/2}})^{T}}{\mathcal{T}}^{1/2}$, which can achieve the optimal two-step convergence when MINRES is used due to the fact that the preconditioned matrix has only two distinct eigenvalues $\pm 1$.

\begin{proposition}\label{Prepo:ideal_P}
	Let $\mathcal{T} \in \mathbb{R}^{mn \times mn} $ be defined by (\ref{eq:main_matrix}). Then, $({({\mathcal{T}^{1/2}})^{T}}{\mathcal{T}}^{1/2})^{-1}\mathcal{Y}\mathcal{T}$ has only $\pm 1$ as eigenvalues.
\end{proposition}
\begin{proof}
	By Lemma (\ref{lemma:real_T}) (a), the preconditioner ${({\mathcal{T}^{1/2}})^{T}}{\mathcal{T}}^{1/2}$ is invertible. By Lemma (\ref{lemma:real_T}) (c), we have $\mathcal{Y} {\mathcal{T}}^{1/2} = {({\mathcal{T}^{1/2}})^{T}} \mathcal{Y}$ which yields $\mathcal{Y} {\mathcal{T}}^{-{1}/{2}} = {({\mathcal{T}^{-1/2}})^{T}} \mathcal{Y}$. Then, we consider the following preconditioned matrix
	\begin{eqnarray*}
		({({\mathcal{T}^{1/2}})^{T}}{\mathcal{T}}^{1/2})^{-1}\mathcal{Y}\mathcal{T} = {\mathcal{T}}^{-{1}/{2}}{({\mathcal{T}^{-1/2}})^{T}} \mathcal{Y}\mathcal{T}&=&{\mathcal{T}}^{-{1}/{2}} \mathcal{Y} {\mathcal{T}}^{-{1}/{2}} \mathcal{T}\\
		&=& {\mathcal{T}}^{-{1}/{2}} \mathcal{Y} {\mathcal{T}}^{-{1}/{2}} {\mathcal{T}^{1/2}} {\mathcal{T}}^{1/2}\\
		&=&{\mathcal{T}}^{-{1}/{2}} \mathcal{Y} {\mathcal{T}}^{1/2}.
	\end{eqnarray*}
	Thus, we have shown that the preconditioned matrix $({({\mathcal{T}^{1/2}})^{T}}{\mathcal{T}}^{1/2})^{-1}\mathcal{Y}\mathcal{T}$ is similar to $\mathcal{Y}$ which is {real-valued, symmetric, and orthogonal}. Hence, it has only $\pm 1$ as eigenvalues.
\end{proof}

In other words, Proposition \ref{Prepo:ideal_P} provides a guide on designing an effective preconditioner for $\mathcal{Y}\mathcal{T}$, since in general ${({\mathcal{T}^{1/2}})^{T}}{\mathcal{T}}^{1/2}$ is computationally expensive to implement. 

In what follows, we shall show that our proposed preconditioner $\mathcal{P}_{\alpha}$ approximates the ideal preconditioner ${({\mathcal{T}^{1/2}})^{T}}{\mathcal{T}}^{1/2}$ in the sense that their difference is a small norm matrix whose magnitude is controlled by the parameter $\alpha$. Its preconditioning effect becomes apparent since $\mathcal{P}_{\alpha}$ gets close to ${({\mathcal{T}^{1/2}})^{T}}{\mathcal{T}}^{1/2}$ as $\alpha$ approaches zero.

Now, we turn our focus on the following lemmas and proposition, which will be used to show our main result. 

For a complex number $z$, denote by $\Re(z)$ and $\Im(z)$, the real part and the imaginary part of $z$, respectively.

Similar to Lemma \ref{lemma:real_T}, the following lemma on $\mathcal{C}_{\alpha}$ is needed.

\begin{lemma}\label{lemma:real_C_alpha}
	Let $\mathcal{C}_{\alpha} \in \mathbb{R}^{mn \times mn} $ be defined by (\ref{def:matrix_C_alpha}) with $0 < \alpha < 1$. Then,
	\begin{description}
		\item[(a)] $\mathcal{C}_{\alpha}^{1/2}$ is a real-valued matrix;
		\item[(b)] $\mathcal{Y} {\mathcal{C}}_{\alpha}^{1/2}$ is symmetric.
	\end{description}
\end{lemma}
\begin{proof}
\cred{By Lemma \ref{calphanonsinglm}, $\mathcal{C}_{\alpha}\in\mathcal{Q}(mn)$. It is well-known that the image $s(\mathbb{C}\setminus(-\infty,0])$ is a subset of $\mathbb{C}_{+}$. Then, by definition of $\mathcal{C}_{\alpha}^{1/2}$ and by the fact that $\sigma(\mathcal{C}_{\alpha})\subset \mathbb{C}\setminus(-\infty,0]$,  we have $\sigma(\mathcal{C}_{\alpha}^{1/2})\subset\mathbb{C}_{+}$. Then, Lemma \ref{psqrtuniqlm} implies that $s(\mathcal{C}_{\alpha})=\mathcal{C}_{\alpha}^{1/2}$. Moreover, $\mathcal{C}_{\alpha}$ is clearly real-valued. Hence, $\mathcal{C}_{\alpha}^{1/2}$ is  real-valued  by part (c) of Lemma \ref{psqrtuniqlm}. Part (a) is shown.}

The proof of part (b) of Lemma \ref{lemma:real_C_alpha} is similar to the proof of part (c) of Lemma \ref{lemma:real_T}. We skip it here.
\end{proof}

\subsection{{Convergence analysis of the proposed PMINRES method}}
In this subsection, we show that the MINRES solver for the preconditioned system \eqref{precsys} has a linear convergence rate independent of both $\tau$ and $h$. 

We first provide the following lemma, which will be used to investigate the convergence rate. 

\cred{For a square matrix ${\bf C}$, denote by $\sigma({\bf C})$ the spectrum of ${\bf C}$. For a real symmetric matrix ${\bf H}$, denote by $\lambda_{\min}({\bf H})$ and $\lambda_{\max}({\bf H})$ the minimal and the maximal eigenvalue of ${\bf H}$, respectively.}

\begin{lemma}\label{minrescvglm}\cite[Theorem 6.13]{ElmanSilvesterWathen2004}
	Let ${\bf P}\in\mathbb{R}^{N\times N}$ and ${\bf A}\in\mathbb{R}^{N\times N}$ be a symmetric positive definite matrix and a symmetric nonsingular matrix, respectively. Suppose $\sigma({\bf P}^{-1}{\bf A})\in[-a_1,-a_2]\cup[a_3,a_4]$ with $a_4\geq a_3>0$, $a_1\geq a_2>0$ and $a_1-a_2=a_4-a_3$. Then, the MINRES solver with ${\bf P}$ as a preconditioner for the linear system ${\bf A}{\bf x}={\bf y}\in\mathbb{R}^{N\times 1}$ has a linear convergence as follows
	\begin{equation*}
		||{\bf r}_{k}||_2\leq 2\left(\frac{\sqrt{a_1a_4}-\sqrt{a_2a_3}}{\sqrt{a_1a_4}+\sqrt{a_2a_3}}\right)^{\lfloor k/2\rfloor}||{\bf r}_{0}||_2,
	\end{equation*}
	where ${\bf r}_{k}={\bf P}^{-1}{\bf y}-{\bf P}^{-1}{\bf A}{\bf x}_k$ denotes the residual vector at the $k$th iteration with ${\bf x}_k$ ($k\geq 1$) being the $k$th iterative solution by MINRES; ${\bf x}_0$ denotes an arbitrary real-valued initial guess; $\lfloor k/2\rfloor$ denotes the integer part of $k/2$.
\end{lemma}

To apply Lemma \ref{minrescvglm} to \eqref{precsys}, we will first need to investigate the spectrum of the preconditioned matrix $\mathcal{P}_{\alpha}^{-1}\mathcal{YT}$. By matrix similarity, we know that $$
\sigma(\mathcal{P}_{\alpha}^{-1}\mathcal{YT})=\sigma(\mathcal{P}_{\alpha}^{-\frac{1}{2}}\mathcal{YT}\mathcal{P}_{\alpha}^{-\frac{1}{2}}).
$$
Hence, it suffices to study the spectral distribution of $\mathcal{P}_{\alpha}^{-\frac{1}{2}}\mathcal{YT}\mathcal{P}_{\alpha}^{-\frac{1}{2}}$. To this end,
we decompose the preconditioned matrix $\mathcal{P}_{\alpha}^{-\frac{1}{2}}\mathcal{YT}\mathcal{P}_{\alpha}^{-\frac{1}{2}}$ as 
\begin{equation}\label{precmatdecompose}
	\mathcal{P}_{\alpha}^{-{1}/{2}}\mathcal{Y}\mathcal{T}\mathcal{P}_{\alpha}^{-{1}/{2}} = \mathcal{ Q}_{\alpha} - \mathcal{E}_\alpha,
\end{equation}
where
\begin{equation*}
	\mathcal{Q}_{\alpha}=\mathcal{P}_{\alpha}^{-{1}/{2}}\mathcal{Y}\mathcal{C}_{\alpha}\mathcal{P}_{\alpha}^{-{1}/{2}},\quad \mathcal{E}_{\alpha}=\mathcal{P}_{\alpha}^{-{1}/{2}}\mathcal{Y}\mathcal{R}_{\alpha}\mathcal{P}_{\alpha}^{-{1}/{2}},\quad \mathcal{R}_{\alpha}:=\mathcal{C}_{\alpha}-\mathcal{T}.
\end{equation*}

In the lemma below, we will show that $\mathcal{Q}_{\alpha}$ is a real symmetric orthogonal matrix (i.e., $\sigma(\mathcal{Q}_{\alpha})\subset\{-1,1\}$) and that $\mathcal{E}_\alpha$ is a small-norm matrix for a properly chosen small $\alpha$.

\begin{lemma}\label{qalphaorthoglm}
	The matrix $\mathcal{Q}_{\alpha}=\mathcal{P}_{\alpha}^{-{1}/{2}}\mathcal{Y}\mathcal{C}_{\alpha}\mathcal{P}_{\alpha}^{-{1}/{2}}$ is both real symmetric and orthogonal, i.e., $\sigma(\mathcal{Q}_{\alpha})\subset\{-1,1\}$.
\end{lemma}
\begin{proof}
	First, we show that $\mathcal{ {Q}_{\alpha}}$ is real-valued. Since both $\mathcal{C}_{\alpha}^{-1/2}$ and $({\mathcal{C}_\alpha^{-1/2}})^{T}$ are real-valued matrices by Lemma \ref{lemma:real_C_alpha} (a), $\mathcal{P}_{\alpha}^{-1}=\mathcal{C}_{\alpha}^{-1/2}{({\mathcal{C}_\alpha^{-1/2}})^{T}}$ is also real-valued. Furthermore, $\mathcal{P}_{\alpha}^{-1}$ itself is a normal equation matrix, so it is HPD. By Lemma \ref{calphanonsinglm}, we know that $\mathcal{P}_{\alpha}^{-1}$ is nonsingular. As a result, $\mathcal{P}_{\alpha}^{-{1}/{2}}$ being a matrix square root of a real-valued HPD matrix is also real-valued HPD. Hence, $\mathcal{ {Q}_{\alpha}}=\mathcal{P}_{\alpha}^{-{1}/{2}}\mathcal{Y} \mathcal{C}_{\alpha}\mathcal{P}_{\alpha}^{-{1}/{2}}$, which is a product of real-valued matrices, is real-valued.

	By Lemma \ref{lemma:real_C_alpha} (a) that ${({\mathcal{C}_\alpha^{1/2}})^{T}}\mathcal{Y}=\mathcal{Y}\mathcal{C}_{\alpha}^{1/2}$ and ${({\mathcal{C}_\alpha^{-1/2}})^{T}}\mathcal{Y}=\mathcal{Y}\mathcal{C}_{\alpha}^{-{1}/{2}}$, we have
	\begin{eqnarray*}
		\mathcal{Q}_{\alpha}^{T}\mathcal{Q}_{\alpha}=	\mathcal{Q}_{\alpha}^2 &=& \mathcal{P}_{\alpha}^{-{1}/{2}}\mathcal{Y}\mathcal{C}_{\alpha}\mathcal{P}_{\alpha}^{-1}\mathcal{Y}\mathcal{C}_{\alpha}\mathcal{P}_{\alpha}^{-{1}/{2}}\\
		&=&  \mathcal{P}_{\alpha}^{-{1}/{2}}\mathcal{Y}\mathcal{C}_{\alpha}\mathcal{C}_{\alpha}^{-{1}/{2}}{({\mathcal{C}_\alpha^{-1/2}})^{T}}\mathcal{Y}\mathcal{C}_{\alpha}\mathcal{P}_{\alpha}^{-{1}/{2}}\\
		&=&\mathcal{P}_{\alpha}^{-{1}/{2}}\mathcal{Y}\mathcal{C}_{\alpha}^{1/2}\mathcal{Y}\mathcal{C}_{\alpha}^{1/2}\mathcal{P}_{\alpha}^{-{1}/{2}}\\
		&=&\mathcal{P}_{\alpha}^{-{1}/{2}}{({\mathcal{C}_\alpha^{1/2}})^{T}}\underbrace{\mathcal{Y}^2}_{=I_{mn}}\mathcal{C}_{\alpha}^{1/2}\mathcal{P}_{\alpha}^{-{1}/{2}}\\
		&=&\mathcal{P}_{\alpha}^{-{1}/{2}}\mathcal{P}_{\alpha}\mathcal{P}_{\alpha}^{-{1}/{2}}\\
		&=&I_{mn}.
	\end{eqnarray*}
	Thus, $\mathcal{ Q}_{\alpha}$ is orthogonal. 
	
	Since $\mathcal{Q}_{\alpha}$ is real symmetric and orthogonal, we know that $\sigma(\mathcal{Q}_{\alpha})\subset \mathbb{R}\cap \{z:|z|=1\} =\{-1,1\}$. The proof is complete.
\end{proof}

In what follows, we will estimate $||\mathcal{E}_{\alpha}||_2$.

From the proof of Lemma \ref{qalphaorthoglm}, we know that $\mathcal{P}_{\alpha}^{-\frac{1}{2}}$ is real symmetric. Moreover, since $\mathcal{R}_{\alpha}$ is a real block Toeplitz matrix with symmetric blocks, $\mathcal{Y}\mathcal{R}_{\alpha}$ is real symmetric. Hence, \cred{$\mathcal{E}_{\alpha}=\mathcal{P}_{\alpha}^{-{1}/{2}}\mathcal{Y}\mathcal{R}_{\alpha}\mathcal{P}_{\alpha}^{-{1}/{2}}$} is real symmetric. Let $\rho(\cdot)$ denote the spectral radius of a matrix. We have
\begin{eqnarray}
	||\mathcal{E}_{\alpha}||_2=	\| \mathcal{P}_{\alpha}^{-{1}/{2}}\mathcal{Y} \mathcal{R}_{\alpha} \mathcal{P}_{\alpha}^{-{1}/{2}} \|_2 &=& \rho( \mathcal{P}_{\alpha}^{-{1}/{2}}\mathcal{Y} \mathcal{R}_{\alpha} \mathcal{P}_{\alpha}^{-{1}/{2}}  )\notag\\
	&=&\rho( \mathcal{P}_{\alpha}^{-1}\mathcal{Y}\mathcal{R}_{\alpha} )\notag\\
	&\leq & \| \mathcal{P}_{\alpha}^{-1}\mathcal{Y}\mathcal{R}_{\alpha} \|_2\notag\\
	&=&||\mathcal{C}_{\alpha}^{-1/2}(\mathcal{C}_{\alpha}^{-1/2})^{T}\mathcal{Y}\mathcal{R}_{\alpha}||_2\notag\\
	&\leq&||\mathcal{C}_{\alpha}^{-1/2}(\mathcal{C}_{\alpha}^{-1/2})^{T}||_2||\mathcal{Y}\mathcal{R}_{\alpha}||_2=||\mathcal{C}_{\alpha}^{-1/2}||_2^2||\mathcal{R}_{\alpha}||_2.\label{ralphaesti1}
\end{eqnarray}

\cred{

\begin{lemma}\label{contiofsqfunc}
The principal square root $s(\cdot)$, as a matrix function, is continuous on $\mathcal{Q}(K)$ for any positive integer $K$.
\end{lemma}
\begin{proof}
By Proposition \ref{psqsmoothprop}, $s(\cdot)$ is holomorphic on $\mathbb{C}\setminus(-\infty,0]$ and thus infinitely differentiable on $\mathbb{C}\setminus(-\infty,0]$. Then, Lemma \ref{contiofsqfunc} is a direct consequence of \cite[Theorem 1.19]{Higham2008}. The proof is complete.
\end{proof}

Define the inverse operation $\iota$ as $\iota:{\bf A}\in GL(K)\rightarrow \iota({\bf A})={\bf A}^{-1}\in GL(K)$ for any positive integer $K$.
\begin{lemma}\label{invopcontilm}
	$\iota(\cdot)$ is a continuous on $\mathcal{Q}(K)$ and  $\iota(\mathcal{Q}(K))\subset \mathcal{Q}(K)$ for any  positive integer $K$ .
\end{lemma}
\begin{proof}
	It is well-known that the inverse operation $\iota(\cdot)$ is continuous  on $GL(K)$. Hence, for  $\iota(\cdot)$ is continuous on  $\mathcal{Q}(K)$, a subset of $GL(K)$.
	
	It remains to show that the image of $\mathcal{Q}(K)$, under the mapping $\iota(\cdot)$, is a subset of $\mathcal{Q}(K)$. For any ${\bf A}\in \mathcal{Q}(K)$, it holds
	\begin{equation*}
		\sigma(\iota({\bf A}))=\sigma({\bf A}^{-1})=\left\{\frac{1}{\lambda}\Big|\lambda\in\sigma({\bf A})\right\}.
	\end{equation*}
	That means, for any $\mu\in\sigma(\iota({\bf A}))$, there exists $z\in\sigma({\bf A})$ such that $\mu=\frac{1}{z}$. Rewrite $z$ as $z=z_r+{\bf i}z_i$ with $z_r,z_i\in\mathbb{R}$. Then,
	\begin{equation*}
		\mu=\frac{1}{z}=\frac{1}{z_r+{\bf i}z_i}=\frac{z_r-{\bf i}z_i}{z_r^2+z_i^2}.
	\end{equation*}
	As ${\bf A}\in\mathcal{Q}(K)$, $z\in\mathbb{C}\setminus(-\infty,0]$ by definition of $\mathcal{Q}(K)$. 
	
	If $z_i=0$, then $z_r+{\bf i}z_i\in \mathbb{C}\setminus(-\infty,0]$ implies that $z_r>0$. In such case $\Re(\mu)=\frac{z_r}{z_r^2+z_i^2}>0$, which means $\mu\in \mathbb{C}\setminus(-\infty,0]$.
	
	If $z_i\neq 0$, then $\Im(\mu)=-\frac{z_i}{z_r^2+z_i^2}\neq 0$, which means $\mu\in \mathbb{C}\setminus(-\infty,0]$.
	
	To conclude, $\mu\in\mathbb{C}\setminus(-\infty,0]$. By the generality of $\mu$ in $\sigma(\iota({\bf A}))$, we know that  $\sigma(\iota({\bf A}))\subset \mathbb{C}\setminus(-\infty,0]$ and thus $\iota({\bf A})\in \mathcal{Q}(K)$. That means for any  ${\bf A}\in \mathcal{Q}(K)$, it holds $\iota({\bf A})\in \mathcal{Q}(K)$. In other words, $\iota(\mathcal{Q}(K))\subset \mathcal{Q}(K)$. The proof is complete.
\end{proof}

\begin{lemma}\label{mainlm}
Denote $r(\alpha):=||s(\mathcal{C}_{\alpha}^{-1})||_2^2$ for $\alpha\in[0,1)$. Then, $r(\cdot)$ is a continuous function on $[0,1)$.
\end{lemma}
\begin{proof}
Clearly, $r(\alpha)$ can be written in a composite form: $r(\alpha)=||s\circ \iota\circ M(\alpha)||_2^2$, where $M(\cdot):\alpha\in[0,1)\rightarrow M(\alpha)=\mathcal{C}_{\alpha}\in\mathcal{Q}(mn)$. To show the continuity of $r(\cdot)$, it suffices to show the continuity of $M(\cdot)$ on $[0,1)$, the continuity of $\iota(\cdot)$ on the image $M([0,1))$, the continuity of $s(\cdot)$ on the image $\iota(M([0,1)))$, and the continuity of $||\cdot||_2^2$ on the image   $s(\iota(M([0,1))))$.

$M(\alpha)$ is an affine operator with respect to $\alpha$ and it is clearly continuous on $[0,1)$.

By Lemma \ref{calphanonsinglm}, $M([0,1))\subset\mathcal{Q}(mn)$. Lemma \ref{invopcontilm} shows that $\iota$ is continuous on $\mathcal{Q}(mn)$ and thus is continuous on the subset $M([0,1))$. 

By Lemma \ref{invopcontilm}, $\iota(M([0,1)))\subset\iota(\mathcal{Q}(mn))\subset\mathcal{Q}(mn)$. Lemma \ref{contiofsqfunc} shows that $s(\cdot)$ is continuous on $\mathcal{Q}(mn)$ and thus is continuous on the subset $\iota(M([0,1)))$.

It is well-known that $||\cdot||_2^2$ is continuous on the whole space $\mathbb{C}^{mn\times mn}$. Hence, $||\cdot||_2^2$ is continuous on $s(\iota(M([0,1))))$.

To conclude, $r(\cdot)$ is continuous on $[0,1)$. The proof is complete.
\end{proof}

\begin{lemma}\label{normcalphainvsqbdlm}
\begin{equation*}
c_0:=\sup\limits_{\alpha\in(0,1/2]}||\mathcal{C}_{\alpha}^{-1/2}||_2^2<+\infty.
\end{equation*}
\end{lemma}
\begin{proof}	
Recall from Lemma \ref{calphanonsinglm} that $\mathcal{C}_{\alpha}\in\mathcal{Q}(mn)$ for each $\alpha\in(0,1)$. Then, by Lemma \ref{invopcontilm}, we have $\mathcal{C}_{\alpha}^{-1}=\iota(\mathcal{C}_{\alpha})\in\mathcal{Q}(mn)$ for each $\alpha\in(0,1)$. Hence, by Lemma \ref{psqrtuniqlm}, we see that there exists a unique matrix ${\bf B}$ such that ${\bf B}^2=\mathcal{C}_{\alpha}^{-1}$ and $\sigma({\bf B})\subset\mathbb{C}_{+}$.

In fact, both $\mathcal{C}_{\alpha}^{-1/2}$ and $s(\mathcal{C}_{\alpha}^{-1})$ are valid candidates of ${\bf B}$. To see this, we firstly note that $(\mathcal{C}_{\alpha}^{-1/2})^2=\mathcal{C}_{\alpha}^{-1}$. Moreover, it is well-known that the image $s(\mathbb{C}\setminus(-\infty,0])$ is a subset of $\mathbb{C}_{+}$. Then, by definition of $\mathcal{C}_{\alpha}^{1/2}$ and by the fact that $\sigma(\mathcal{C}_{\alpha})\subset \mathbb{C}\setminus(-\infty,0]$, we see that $\sigma(\mathcal{C}_{\alpha}^{1/2})\subset \mathbb{C}_{+}$. That means for each $\mu\in\sigma(\mathcal{C}_{\alpha}^{-1/2})$, there exists $\lambda\in\sigma(\mathcal{C}_{\alpha}^{1/2})\subset \mathbb{C}_{+}$ such that $\mu=\frac{1}{\lambda}$. Rewrite $\lambda$ as $\lambda=\lambda_r+{\bf i}\lambda_i$ with $\lambda_r,\lambda_i\in\mathbb{R}$. Then,
\begin{equation*}
\mu=\frac{1}{\lambda}=\frac{\lambda_r-{\bf i}\lambda_i}{\lambda_r^2+\lambda_i^2}.
\end{equation*}
Since $\lambda\in\mathbb{C}_{+}$, $\lambda_r>0$. Hence, $\Re(\mu)=\frac{\lambda_r}{\lambda_r^2+\lambda_i^2}>0$. That means $\mu\in\mathbb{C}_{+}$. By the generality of $\mu$ in $\sigma(\mathcal{C}_{\alpha}^{-1/2})$, we know that $\sigma(\mathcal{C}_{\alpha}^{-1/2})\subset \mathbb{C}_{+}$. Hence, $\mathcal{C}_{\alpha}^{-1/2}$ is a valid candidate of ${\bf B}$.

On the other hand, by Lemma \ref{psqrtuniqlm}, we know that $s(\mathcal{C}_{\alpha}^{-1})$ is also a valid candidate of matrix ${\bf B}$. Then, the uniqueness of ${\bf B}$ implies that  $s(\mathcal{C}_{\alpha}^{-1})=\mathcal{C}_{\alpha}^{-1/2}$ for each $\alpha\in(0,1)$.

Therefore, $r(\alpha)=||s(\mathcal{C}_{\alpha}^{-1})||_2^2=||\mathcal{C}_{\alpha}^{-1/2}||_2^2$ for each $\alpha\in(0,1)$. Therefore,
\begin{equation*}
\sup\limits_{\alpha\in(0,1/2]}||\mathcal{C}_{\alpha}^{-1/2}||_2^2=\sup\limits_{\alpha\in(0,1/2]}r(\alpha)\leq \sup\limits_{\alpha\in[0,1/2]}r(\alpha).
\end{equation*}
As shown in Lemma \ref{mainlm}, we know that $r(\cdot)$ is continuous on the compact set $[0,1/2]$. A continuous function on a compact set has a finite upper bound. Therefore,
\begin{equation*}
\sup\limits_{\alpha\in(0,1/2]}||\mathcal{C}_{\alpha}^{-1/2}||_2^2\leq  \sup\limits_{\alpha\in[0,1/2]}r(\alpha)<+\infty.
\end{equation*}
The proof is complete.
\end{proof}


}

\begin{definition}\label{continuoufunc}
A function $g:\mathbb{X}\rightarrow\mathbb{Y}$ between two topological spaces $\mathbb{X}$ and $\mathbb{Y}$ is said to be continuous if and only if for every open set $\mathbb{V}\subset\mathbb{Y}$, the inverse image $$g^{-1}(\mathbb{V}):=\{X\in\mathbb{X}|g(X)\in\mathbb{V}\},$$ is an open subset of $\mathbb{X}$.
\end{definition}
\begin{lemma}\label{eigenperturblm}\cite[Corollary 6.3.4]{horn2012matrix}
	For $N\times N$ Hermitian matrices ${\bf H}$ and ${\bf E}$, if $\hat{\lambda}$ is an eigenvalue of ${\bf H}+{\bf E}$, then there exists an eigenvalue $\lambda$ of ${\bf H}$ such that
	\begin{equation*}
		|\lambda-\hat{\lambda}|\leq ||{\bf E}||_2.
	\end{equation*} 
\end{lemma}

The following lemma indicates that the spectrum of the preconditioned matrix $\mathcal{P}_{\alpha}^{-1}\mathcal{YT}$ is located in a disjoint interval excluding the origin.

\cred{
\begin{lemma}\label{precedmatspectrmlm}
For $\alpha\in(0,1/2]$, we have
	\begin{equation*}
		\sigma(\mathcal{P}_{\alpha}^{-1}\mathcal{YT})\subset \left[-1-\alpha c_1,-1+\alpha c_1\right]\cup \left[1-\alpha c_1,1+\alpha c_1\right],
	\end{equation*}
	where $c_1>0$ independent of $\alpha$ is given as follows
	\begin{equation*}
	c_1=c_0\left(\sqrt{||L||_2^2+1}+1\right),
	\end{equation*} 
	with $c_0$ given in Lemma \ref{normcalphainvsqbdlm}.
\end{lemma}
\begin{proof}
	By matrix similarity and \eqref{precmatdecompose}, we have $\sigma(\mathcal{P}_{\alpha}^{-1}\mathcal{YT})=\sigma(\mathcal{P}_{\alpha}^{-{1}/{2}}\mathcal{Y}\mathcal{T}\mathcal{P}_{\alpha}^{-{1}/{2}})=\sigma(\mathcal{Q}_{\alpha}-\mathcal{E}_{\alpha})$. Since both $\mathcal{Q}_{\alpha}$ and $-\mathcal{E}_{\alpha}$ are Hermitian, Lemma \ref{eigenperturblm} is applicable. In other words, for $\hat{\lambda}\in\sigma(\mathcal{Q}_{\alpha}-\mathcal{E}_{\alpha})$, there exists $\lambda\in\sigma(\mathcal{Q}_{\alpha})$ such that $|\hat{\lambda}-\lambda|\leq ||\mathcal{E}_{\alpha}||_2$. By Lemma \ref{qalphaorthoglm}, we know that either $\lambda=1$ or $\lambda=-1$. That means either $|\hat{\lambda}-1|\leq ||\mathcal{E}_{\alpha}||_2$ or  $|\hat{\lambda}+1|\leq ||\mathcal{E}_{\alpha}||_2$. Thus, we have
	\begin{eqnarray}\label{precmatspectrumesti1}
		\sigma(\mathcal{P}_{\alpha}^{-1}\mathcal{YT})&=&\sigma(\mathcal{Q}_{\alpha}-\mathcal{E}_{\alpha})\\\nonumber
		&\subset&\left[-1-||\mathcal{E}_{\alpha}||_2,-1+||\mathcal{E}_{\alpha}||_2\right]\cup \left[1-||\mathcal{E}_{\alpha}||_2,1+||\mathcal{E}_{\alpha}||_2\right].
	\end{eqnarray}
	Hence, it suffices to show that $||\mathcal{E}_{\alpha}||_2\leq c_1\alpha$. By \eqref{ralphaesti1} and Lemma \ref{normcalphainvsqbdlm}, we have
	\begin{equation*}
	||\mathcal{E}_{\alpha}||_2\leq ||\mathcal{C}_{\alpha}^{-1/2}||_2^2||\mathcal{R}_{\alpha}||_2\leq c_0||\mathcal{R}_{\alpha}||_2.
	\end{equation*}
	Notice that
	\begin{equation*}
	\mathcal{R}_{\alpha}=\alpha\mathcal{R}_{1},\quad \mathcal{R}_{1}=\left[\begin{array}[c]{ccccc}
	&&&L&-2I_m\\
	&&&&L\\
	&&&&\\
	&&&&\\
	&&&&
	\end{array}\right].
	\end{equation*}
	Therefore, it is straightforward to check that
	\begin{align*}
	||\mathcal{R}_{\alpha}||_2=\alpha||\mathcal{R}_{1}||_2=\alpha\lambda_{\max}\left(\left[\begin{array}[c]{cc}
	L&-2I_m\\
	&L
	\end{array}\right]\left[\begin{array}[c]{cc}
	L&\\
	-2I_m&L
	\end{array}\right]\right)=\alpha\left(\sqrt{||L||_2^2+1}+1\right).
	\end{align*}
	Then,
	\begin{equation*}
		||\mathcal{E}_{\alpha}||_2\leq c_0	||\mathcal{R}_{\alpha}||_2=\alpha c_0\left(\sqrt{||L||_2^2+1}+1\right)=c_1\alpha.
	\end{equation*}
	The proof is complete.
\end{proof}

}

\begin{theorem}\label{maincvgthm}
	For any constant $\nu\in(0,1)$, choose $\alpha\in(0,\eta]$ with
	$$
	\cred{\eta:=\min\left\{\frac{\nu^2}{c_1},\frac{1}{2}\right\},}
	$$
	where $c_1$ is defined in Lemma \ref{precedmatspectrmlm}.
	Then, the MINRES solver for the preconditioned system (\ref{precsys}) has a convergence rate independent of both $\tau$ and $h$, i.e., 
	\begin{equation*}
		||\tilde{\bf r}_k||_2\leq 2\nu^{k-1}||\tilde{\bf r}_0||_2, \quad k\geq 1,
	\end{equation*}
	where $\tilde{\bf r}_k=\mathcal{P}_{\alpha}^{-1}\mathcal{YT}{\bf u}_k-\mathcal{P}_{\alpha}^{-1}\mathcal{Y}{\bf f}$ denotes the residual vector at $k$th MINRES iteration with ${\bf u}_k$ denoting the $k$th iterative solution by MINRES; $\tilde{\bf r}_0$ denotes the initial residual vector computed by an arbitrary real-valued initial guess ${\bf u}_0$.
\end{theorem}

\cred{
\begin{proof}
	By $\alpha\in(0,\eta]\subset(0,1/2]$ and Lemma \ref{precedmatspectrmlm}, we know that
	\begin{align*}
		\sigma(\mathcal{P}_{\alpha}^{-1}\mathcal{YT})&\subset \left[-1-\alpha c_1,-1+\alpha c_1\right]\cup \left[1-\alpha c_1,1+\alpha c_1\right]\notag\\
		&\subset \left[-1-\frac{c_1\nu^2 }{c_1},-1+\frac{c_1\nu^2 }{c_1}\right]\cup \left[1-\frac{c_1\nu^2 }{c_1},1+\frac{c_1\nu^2 }{c_1}\right]\notag\\
		&=\left[-1-\nu^2,-1+\nu^2\right]\cup \left[1-\nu^2,1+\nu^2\right],
	\end{align*}
	which together with Lemma \ref{minrescvglm} implies that
	\begin{align*}
		||\tilde{\bf r}_k||_2&\leq 2\left(\frac{\sqrt{(1+\nu^2)^2}-\sqrt{(1-\nu^2)^2}}{\sqrt{(1+\nu^2)^2}+\sqrt{(1-\nu^2)^2}}\right)^{\lfloor k/2\rfloor}||\tilde{\bf r}_0||_2\notag\\
		&= 2\left(\nu^2\right)^{\lfloor k/2\rfloor}||\tilde{\bf r}_0||_2\notag\\
		&\leq 2\left(\nu^2\right)^{\frac{k}{2}-\frac{1}{2}}||\tilde{\bf r}_0||_2=2\nu^{k-1}||\tilde{\bf r}_0||_2.
	\end{align*}
	The proof is complete.
\end{proof}

\begin{remark}\label{remark_finalfinal}
As indicated in Theorem \ref{maincvgthm}, the convergence rate bound (i.e., $\nu$) of the MINRES solver can be taken as a constant independent of matrix size. For such a fixed $\nu$, the corresponding $\eta$ needs to be of $\mathcal{O}(\frac{1}{c_1})$, where $c_1$ defined in Lemma \ref{precedmatspectrmlm} depends on both $c_0$ and $||L||_2$. Clearly, $||L||_2$ depends on $m$ (the number of spatial grid points). Meanwhile, $c_0$ defined in Lemma \ref{normcalphainvsqbdlm}  depends on $n$ (the number of time steps). To summarize, setting $\eta = \min{\{\nu^2/{c_1}, 1/2}\}$ for some constant $\nu \in (0,1)$ and taking $\alpha\in(0,\eta]$ leads to a matrix-size independent convergence rate bound of the MINRES solver. 
\end{remark}
}

\subsection{Implementations}\label{subsec:implementationP}

We begin by discussing the computation of $\mathcal{Y}\mathcal{T}\mathbf{v}$ for any given vector $\mathbf{v}$. Since $\mathcal{T}$ is a sparse matrix, computing the matrix-vector product $\mathcal{T}\mathbf{v}$ only requires linear complexity of $\mathcal{O}(mn)$. As the action of $\mathcal{Y}$ is a simple reordering of entries that poses virtually no work, computing $\mathcal{Y}\mathcal{T}\mathbf{v}$ needs the same complexity.  Alternatively, due to the fact that $\mathcal{T}$ itself is a block Toeplitz matrix, it is well-known that $\mathcal{T}\mathbf{v}$ (and hence $\mathcal{Y}\mathcal{T}\mathbf{v}$) can be computed in $\mathcal{O}(mn \log{n})$ operations using fast Fourier transforms and the required storage is of $\mathcal{O}(mn)$.

Now, we discuss how to efficiently invert $\mathcal{P}_{\alpha}$, based on similar arguments discussed in, for instance, \cite[Section 4]{doi:10.1137/20M1316354}. The invertibility of $\mathcal{P}_{\alpha}$ is guaranteed for $\alpha\in(0,1)$, which will be proven in the next section. In each iteration of MINRES, it is required to compute $\mathcal{P}_{\alpha}^{-1}\mathbf{y}=\mathcal{C}_{\alpha}^{-{1}/{2}}{({\mathcal{C}_\alpha^{-1/2}})^{T}}\mathbf{y}$ for any given vector $\mathbf{y}$. 

Note that $\mathcal{C}_{\alpha}$ is composed of the $\alpha$-circulant matrices $B_j^{(\alpha)} = D_{\alpha}^{-1}\mathbb{F}^{*}\cred{\Lambda_{\alpha,j}}\mathbb{F}D_{\alpha},~j=1,2,$ where \cred{$\Lambda_{\alpha,j}={\rm diag}(\lambda_{j,k}^{(\alpha)})_{k=1}^{n}$ with $\lambda_{1,k}^{(\alpha)}=1+\alpha^{\frac{2}{n}}\theta_n^{2(k-1)}$ and $\lambda_{2,k}^{(\alpha)}=-2\alpha^{\frac{1}{n}}\theta_n^{k-1}.$ Also,} $D_{\alpha}={\rm diag}(\alpha^{\frac{i-1}{n}})_{i=1}^{n}$ and $\mathbb{F}=\frac{1}{\sqrt{n}}[\theta_n^{(i-1)(j-1)}]_{i,j=1}^{n}$ with $\theta_n =\exp(\frac{2\pi {\bf i}}{n})$ and ${\bf i}=\sqrt{-1}$. \cred{Considering (\ref{eqn:C_decom}) and} the eigendecomposition of $L=U \Lambda U^T$, we have
\begin{align*}
	{({\mathcal{C}_\alpha^{-1/2}})^{T}}&= { \big( \big(B_1^{(\alpha)} \otimes L + B_2^{(\alpha)} \otimes (-2I_{m}) \big)^{-1/2} \big)^T }\\
	&= { \big( \big( D_{\alpha}^{-1}\mathbb{F}^{*} \cred{\Lambda_{\alpha,1}}\mathbb{F}D_{\alpha} \otimes U \Lambda U^T + D_{\alpha}^{-1}\mathbb{F}^{*} \cred{\Lambda_{\alpha,2}}\mathbb{F}D_{\alpha} \otimes (-2I_{m}) \big)^{-1/2} \big)^{T} }\\
	&=  (D_{\alpha}^{-1}\mathbb{F}^{*} \otimes U)^{-T} { \big( \big(\cred{\Lambda_{\alpha,1}}\otimes \Lambda + \cred{\Lambda_{\alpha,2}}\otimes (-2I_{m}) \big)^{-1/2} \big)^{T} }(\mathbb{F}D_{\alpha} \otimes U^T )^{-T}.
\end{align*}
Thus, the product ${({\mathcal{C}_\alpha^{-1/2}})^{T}}\mathbf{y}$ for a given vector $\mathbf{y}$ can be implemented via the following three steps:
\begin{eqnarray*}
	&&\textrm{Step 1: Compute}~\widetilde{\mathbf{y}} =  (\mathbb{F}D_{\alpha} \otimes U^T )^{-T}\mathbf{y}; \\
	&&\textrm{Step 2: Compute}~ \widetilde{\mathbf{z}} =  { \big( \big(\cred{\Lambda_{\alpha,1}} \otimes \Lambda + \cred{\Lambda_{\alpha,2}} \otimes (-2I_{m}) \big)^{-1/2} \big)^{T} } \widetilde{\mathbf{y}};\\
	&&\textrm{Step 3: Compute}~{\mathbf{z}} =  (D_{\alpha}^{-1}\mathbb{F}^{*}\otimes U)^{-T} \widetilde{\mathbf{z}}.
\end{eqnarray*}

When the spatial grid is uniformly partitioned, \textcolor{black}{the orthogonal matrix $U$ becomes a (multi-dimension) discrete sine transform matrix $\mathbb{S}$.} In this case, both Steps 1 \& 3 can be computed efficiently via fast Fourier transforms and fast sine transforms in $\mathcal{O}(mn\log{n})$ operations. As for Step 2, the required computations take $\mathcal{O}(n m)$ operations since the matrix involved is a simple diagonal matrix.

In effect, the overall product $\mathcal{P}_{\alpha}^{-1}\mathbf{y}=\mathcal{C}_{\alpha}^{-{1}/{2}}{({\mathcal{C}_\alpha^{-1/2}})^{T}}\mathbf{y}$ can be computed by repeating the above three-step procedure two times. 

In the case where the underlying differential equation involves variable coefficients, the implementation of our proposed preconditioner can be carried out in a similar way as in the constant-coefficient case. A related discussion will be given in Section \ref{sec:numerical}.

\cred{	
Certainly, when $\alpha$ is small, computations involving $D_{\alpha}^{-1}$ may introduce stability concerns as extensively studied in \cite{Gander_Halpern_etc_2019,Wu2018}. However, as examined in \cite{doi:10.1137/19M1309869}, it appears that this does not significantly impact convergence when a Krylov subspace method is used with a properly chosen small $\alpha$ independent of the meshes. The stability implications of small $\alpha$ warrant further investigation, which we will address in our future research.}
	
	\section{Numerical examples}\label{sec:numerical}
	
	
	In this section, we demonstrate the effectiveness of our proposed preconditioner. All numerical experiments are carried out using {Octave 8.2.0 on a Dell R640 server equipped with dual Xeon Gold 6246R 16-Cores 3.4GHz CPUs, 512GB RAM running Ubuntu 20.04 LTS}. The CPU time in seconds is measured using the Octave {built-in} functions \textbf{tic/toc}. Both Steps 1 \& 3 in Section \ref{subsec:implementationP} are implemented by the functions \textbf{dst} (discrete sine transform) and \textbf{fft} (fast Fourier transform). Furthermore, \cred{the MINRES solver used is implemented using the {built-in} function \textbf{minres} on Octave}. We choose a zero initial guess and a stopping tolerance of $10^{-6}$ based on the reduction in relative residual norms for the MINRES solver tested unless otherwise indicated. 
	
	We adopt the notation MINRES-$\mathcal{P}_{\alpha}$ to represent the MINRES solver with the proposed preconditioner $\mathcal{P}_{\alpha}$.  We will test our proposed MINRES-$\mathcal{P}_{\alpha}$ method for  wave equation compare its performance against MINRES-$\mathcal{P}_{1}$ proposed in \cite{doi:10.1137/16M1062016} and MINRES-$I_{mn}$ (i.e., the non-preconditioned case). We set $\alpha$ to $10^{-6}$ for MINRES-$\mathcal{P}_{\alpha}$, a value empirically chosen based on a series of experiments aimed at optimizing performance. The stopping criterion of for all MINRES solvers tested in this section is set as $\|{\bf r}_k \|_2\leq 10^{-6}$, where ${\bf r}_k$ denotes the residual vector at $k$-th iteration step. 
	
	Denote by `Iter', the iteration numbers for solving a linear system by an iterative solver. Denote by `CPU', the computational time in unit of second. Denote by `DoF', the degree of freedom, i.e., the number of unknowns. For ease of statement, we take the same spatial step size $h$ for each spatial direction in all examples appearing in this section. We define the error as
	{
		\begin{equation*}
			{\rm E}_{\tau,h}:=\max\limits_{1\leq k\leq n}h^{d/2}\|({\bf u}^{(k)}_{\rm{iter}}-u(\mathcal{G}_h,k\tau)\|_2,
		\end{equation*}
		where $d$ is the spatial dimension, $h$ denotes the spatial step size along each spatial direction, ${\bf u}^{(k)}_{\rm{iter}}$ denotes the iterative solution of the linear system at $k$-th time level, and $u(\mathcal{G}_h,k\tau)$ denotes the values of exact solution $u(\cdot,k\tau)$ on the spatial grids.
	}
	
	{In the following two examples, as discussed, the concerned wave equations is discretized using the finite difference scheme proposed in \cite{LiLiuXiao2015}, which was shown to be unconditionally convergent with a second-order accuracy without imposing the CFL condition.}
	\begin{example}\label{examplewveq}
		{\rm
			In this example, we consider a wave equation \eqref{eqn:wave} with
			\begin{align*}
				&d=2,~\Omega=(0,1)\times (0,1),~T=1,~-\psi_1({ x})=\psi_0({ x})=x_1(x_1-1)x_2(x_2-1),\\
				&f({ x},t)=\exp(-t)(x_1({x_1}-1)x_2(x_2-1)-2(x_1(x_1-1)+x_2(x_2-1))).
			\end{align*}
			The exact solution of Example \ref{examplewveq} is given by 
			\begin{equation*}
				u({ x},t)=\exp(-t)x_1(x_1-1)x_2(x_2-1).
			\end{equation*}
			Applying the solver aforementioned for Example \ref{examplewveq}, the corresponding numerical results are presented.
			{Table \ref{explwveqtbl_con_minres_itercpu} shows} that (i) MINRES-$\mathcal{P}_{\alpha}$ is considerably more efficient than MINRES-$\mathcal{P}_{1}$ in terms of computational time and iteration number, while the accuracy of both solvers are roughly the same; (ii) the convergence rate (i.e., the iteration number) of MINRES-$\mathcal{P}_{\alpha}$ does not deteriorate as the temporal or spatial girds get refined. Evidently, the optimality of the proposed preconditioner is demonstrated, namely, the complexity of the proposed MINRES-$\mathcal{P}_{\alpha}$ is merely proportional to the complexity of the matrix-vector multiplication. 
			
			{As shown in Table \ref{explwveqtbl_con_minres_errors}, the errors by all solvers adopted are roughly the same. In fact, the errors of all solvers tested in this section are always roughly the same in each example. For this reason, we will skip the results of ${\rm E}_{\tau,h}$ in the subsequent example.}
			
			\begin{table}[H]
				{
					\begin{center}
						\caption{Iteration numbers and CPU times of MINRES-$\mathcal{P}_{\alpha}$,  MINRES-$\mathcal{P}_1$, and MINRES-$I_{mn}$ for Example \ref{examplewveq}.}\label{explwveqtbl_con_minres_itercpu}
						\setlength{\tabcolsep}{0.8em}
						{\small
							\begin{tabular}[c]{ccc|cc|cc|cc}
								\hline
								\multirow{2}{*}{$\tau$} &\multirow{2}{*}{$h$} &\multirow{2}{*}{DoF}& \multicolumn{2}{c|}{MINRES-$\mathcal{P}_{\alpha}$} &\multicolumn{2}{c}{MINRES-$\mathcal{P}_1$}&\multicolumn{2}{|c}{MINRES-$I_{mn}$}\\
								\cline{4-9}
								&&&$\mathrm{Iter}$&$\mathrm{CPU(s)}$ &$\mathrm{Iter}$&$\mathrm{CPU(s)}$ &$\mathrm{Iter}$&$\mathrm{CPU(s)}$\\
								\hline
								\multirow{4}{*}{$2^{-4}$}
								&$2^{-4}$ &3600     &2   &\cred{0.0059}   &140  &0.21    &614    &0.12  \\
								&$2^{-5}$ &15376    &2   &\cred{0.015}    &87   &0.31    &1992   &2.10  \\
								&$2^{-6}$ &63504    &2   &\cred{0.030}    &223  &1.97    &7659   &44.38 \\
								&$2^{-7}$ &258064   &2   &\cred{0.089}    &476  &12.12   &30491  &308.83\\
								\hline
								\multirow{4}{*}{$2^{-5}$}
								&$2^{-4}$ &7200     &2   &\cred{0.0089}   &146  &0.27    &1204   &0.27  \\
								&$2^{-5}$ &30752    &2   &\cred{0.019}    &332  &1.64    &2574   &5.13  \\
								&$2^{-6}$ &127008   &2   &\cred{0.049}    &710  &9.35    &7897   &57.41 \\
								&$2^{-7}$ &516128   &2   &\cred{0.16}     &797  &37.34   &29386  &564.35\\
								\hline
								\multirow{4}{*}{$2^{-6}$}
								&$2^{-4}$ &14400    &2   &\cred{0.013}    &198  &0.67    &3388   &2.75  \\
								&$2^{-5}$ &61504    &2   &\cred{0.026}    &556  &4.49    &4695   &25.13  \\
								&$2^{-6}$ &254016   &2   &\cred{0.084}    &1567 &36.79   &9331   &84.19 \\
								&$2^{-7}$ &1032256  &2   &\cred{0.31}     &6498 &625.95  &27664  &1138.95\\
								\hline
								\multirow{4}{*}{$2^{-7}$}
								&$2^{-4}$ &28800    &2   &\cred{0.017}    &236  &1.17    &11632   &22.26  \\
								&$2^{-5}$ &123008   &2   &\cred{0.049}    &804  &10.19   &13062   &87.66  \\
								&$2^{-6}$ &508032   &2   &\cred{0.014}    &3334 &151.62  &17164   &314.98  \\
								&$2^{-7}$ &2064512  &2   &\cred{0.83}     &8445 &2233.43 &34255   &2941.73 \\
								\hline
						\end{tabular}}
					\end{center}
				}
			\end{table}
			
			\begin{table}[H]
				{
					\begin{center}
						\caption{Errors of MINRES-$\mathcal{P}_{\alpha}$,  MINRES-$\mathcal{P}_1$, and MINRES-$I_{mn}$ for Example \ref{examplewveq}.}\label{explwveqtbl_con_minres_errors}
						\setlength{\tabcolsep}{0.8em}
						{\small
							\begin{tabular}{ccc|ccc}
								\hline
								\multirow{2}{*}{$\tau$}   & \multirow{2}{*}{$h$} & \multirow{2}{*}{DoF} & \multicolumn{3}{c}{${\rm E}_{\tau,h}$}                                                 \\ \cline{4-6} 
								&                      &                      & MINRES-$\mathcal{P}_{\alpha}$ & MINRES-$\mathcal{P}_{1}$ & MINRES-$I_{mn}$ \\ \hline
								\multirow{4}{*}{$2^{-4}$} & $2^{-4}$             & 3600                 & \cred{3.04e-4}                                             & 3.04e-4                  & 3.04e-4         \\
								& $2^{-5}$             & 15376                & \cred{3.04e-4}                       & 3.05e-4                  & 3.05e-4         \\
								& $2^{-6}$             & 63504                & \cred{3.05e-4}                       & 3.05e-4                  & 3.05e-4         \\
								& $2^{-7}$             & 258064               & \cred{3.05e-4}                       & 3.05e-4                  & 3.05e-4         \\ \hline
								\multirow{4}{*}{$2^{-5}$} & $2^{-4}$             & 7200                 & \cred{7.67e-5}                       & 7.70e-5                  & 7.70e-5         \\
								& $2^{-5}$             & 30752                & \cred{7.68e-5}                       & 7.71e-5                  & 7.71e-5         \\
								& $2^{-6}$             & 127008               & \cred{7.69e-5}                       & 7.71e-5                  & 7.71e-5         \\
								& $2^{-7}$             & 516128               & \cred{7.69e-5}                       & 7.71e-5                  & 7.71e-5         \\ \hline
								\multirow{4}{*}{$2^{-6}$} & $2^{-4}$             & 14400                & \cred{1.87e-5}                                                                   & 1.93e-5                  & 1.93e-5         \\
								& $2^{-5}$             & 61504                & \cred{1.88e-5}                                             & 1.93e-5                  & 1.93e-5         \\
								& $2^{-6}$             & 254016               & \cred{1.88e-5}                                                                   & 1.93e-5                  & 1.93e-5         \\
								& $2^{-7}$             & 1032256              & \cred{1.88e-5}                                             & 1.93e-5                  & 1.93e-5         \\ \hline
								\multirow{4}{*}{$2^{-7}$} & $2^{-4}$             & 28800                & \cred{3.62e-6}                                                                  & 4.83e-6                  & 4.89e-6         \\
								& $2^{-5}$             & 123008               & \cred{3.62e-6}                                             & 4.84e-6                  & 4.99e-6         \\
								& $2^{-6}$             & 508032               & \cred{3.63e-6}                                             & 4.84e-6                  & 5.03e-6         \\
								& $2^{-7}$             & 2064512              & \cred{3.63e-6}                       & 4.84e-6                  & 5.06e-6         \\ \hline
							\end{tabular}
						}
					\end{center}
				}
			\end{table}

		}
	\end{example}
	
	For the wave equation with variable diffusion coefficients given below in \eqref{eqn: var coeff wave}, we adopt the same discretization scheme as the one used in Section \ref{sec:intro}.
	\begin{equation}\label{eqn: var coeff wave}
		\left\{
		\begin{array}{lc}
			u_{tt}({ x},t) = \nabla(a({ x})\nabla u({ x},t)) + f({ x},t), & (x,t)\in \Omega \times (0,T], \\
			u = 0, & ({ x},t)\in \partial \Omega \times (0,T], \\
			u({ x},0)=\psi_0,~u_t({ x},0)=\psi_1, & x \in \Omega\subset\mathbb{R}^{d}.
		\end{array}
		\right.\,
	\end{equation} 
	where $a({ x})$ is a known positive function. Similar to \eqref{eq:main_matrix}, the following time-space linear system is obtained from \eqref{eqn: var coeff wave}:
	
	\begin{equation}\label{tssystemvarwaveeq}
		\begin{bmatrix}
			L_{a} &   &  &  & \\
			-2I_{m}  & L_{a}    & & & \\
			L_{a} &  -2I_{m}  & L_{a}  & &  \\
			&   \ddots & \ddots & \ddots &  \\
			&  &  L_{a}  & -2I_{m} & L_{a}
		\end{bmatrix}
		\left[\begin{array}{c}\mathbf{u}^{(1)}\\ \mathbf{u}^{(2)}\\ \mathbf{u}^{(3)}\\ \vdots \\ \mathbf{u}^{(n)}
		\end{array} \right]
		= \tau^2 \left[\begin{array}{c} \mathbf{f}^{(0)}/2+ \Psi_1/\tau + \Psi_0/\tau^2  \\ 
			\mathbf{f}^{(1)} - L \Psi_0 /\tau^2\\
			\mathbf{f}^{(2)} \\
			\vdots \\
			\mathbf{f}^{(n-1)}
		\end{array} \right],
	\end{equation}
	where $L_{a}=I_m-\frac{\tau^2}{2}\Delta_{a,h}\in\mathbb{R}^{m\times m}$ and $\Delta_{a,h}$ is the central difference discretization of $ \nabla a({ x})\nabla$. Clearly, the proposed preconditioner for the linear system \eqref{tssystemvarwaveeq} is not fast invertible since the spatial matrix cannot be diagonalized by the fast sine transform. To remedy this situation, we consider replacing the diffusion coefficient function $a$ with $\bar{a}$ (i.e., the mean value of $a$ on the spatial grid) and hence the time space matrix becomes
	\begin{equation}\label{mean-coefftsmat}
		\begin{bmatrix}
			L_{\bar{a}} &   &  &  & \\
			-2I_{m}  & L_{\bar{a}}    & & & \\
			L_{\bar{a}} &  -2I_{m}  & L_{\bar{a}}  & &  \\
			&   \ddots & \ddots & \ddots &  \\
			&  &  L_{\bar{a}}  & -2I_{m} & L_{\bar{a}}
		\end{bmatrix},
	\end{equation}
	where $L_{\bar{a}}=I_m-\frac{\bar{a}\tau^2}{2}\Delta_{h}\in\mathbb{R}^{m\times m}$ with $\Delta_{h}$ denoting the discretization matrix for the constant  Laplacian operator. Then, both our proposed preconditioner $\mathcal{P}_{\alpha}$ and $\mathcal{P}_1$ for the matrix \eqref{mean-coefftsmat} are fast invertible. Thus, we use them as the preconditioners for \eqref{tssystemvarwaveeq} and test their performance on Example \ref{varwaveeq}.
	
	\begin{example}\label{varwaveeq}
		{\rm 
			In this example, we consider the wave equation \eqref{eqn: var coeff wave} with 
			\begin{align*}
				&d=2,~\Omega=(0,1)\times(0,1),~T=1,~\psi_1({ x})=\psi_0({ x})=x_2(1-x_2)x_1(1-x_1),\\
				&a({ x})=(30+\sin(x_1)^2)(30+\sin(x_2)^2),\\
				&f({ x},t)=\exp(t)[x_1(1-x_1)x_2(1-x_2)-\sin(2x_1)(30+\sin(x_2)^2)(1-2x_1)x_2(1-x_2)\\
				&\quad\quad\quad-\sin(2x_2)(30+\sin(x_1)^2)(1-2x_2)x_1(1-x_1)+2a({ x})(x_1(1-x_1)+x_2(1-x_2))].
			\end{align*}
			
			We apply MINRES-$\mathcal{P}_{\alpha}$ and MINRES-$\mathcal{P}_1$ for Example \ref{varwaveeq}. {Tables \ref{explvarwveqtbl_var_wave_itercpu} shows} that (i) MINRES-$\mathcal{P}_{\alpha}$ is much more efficient than MINRES-$\mathcal{P}_{1}$ in terms of iteration number and computational time while accuracy of the both solvers are roughly the same; (ii) iteration number of MINRES-$\mathcal{P}_{\alpha}$ is stable with respect to changes of $\tau$ and $h$.

			\begin{table}[H]
				{
					\begin{center}
						\caption{Iteration numbers and CPU times of MINRES-$\mathcal{P}_{\alpha}$, MINRES-$\mathcal{P}_1$, and MINRES-$I_{mn}$ for Example \ref{varwaveeq}.}\label{explvarwveqtbl_var_wave_itercpu}
						\setlength{\tabcolsep}{0.8em}
						{\small
							\begin{tabular}[c]{ccc|cc|cc|cc}
								\hline
								\multirow{2}{*}{$\tau$} &\multirow{2}{*}{$h$} &\multirow{2}{*}{DoF}& \multicolumn{2}{c|}{MINRES-$\mathcal{P}_{\alpha}$} &\multicolumn{2}{|c}{MINRES-$\mathcal{P}_1$}&\multicolumn{2}{|c}{MINRES-$I_{mn}$}\\
								\cline{4-9}
								&&&$\mathrm{Iter}$&$\mathrm{CPU(s)}$ &$\mathrm{Iter}$&$\mathrm{CPU(s)}$ &$\mathrm{Iter}$&$\mathrm{CPU(s)}$\\
								\hline
								\multirow{4}{*}{$2^{-4}$}
								&$2^{-4}$ &3600     &8   &\cred{0.015}    &415  &0.54    &3046    &0.62  \\
								&$2^{-5}$ &15376    &8   &\cred{0.032}    &1719 &6.31    &12891   &12.84 \\
								&$2^{-6}$ &63504    &8   &\cred{0.084}    &7100 &62.41   &52655   &301.21\\
								&$2^{-7}$ &258064   &8   &\cred{0.24}     &28911&711.78  &$>$200000 &-\\
								\hline
								\multirow{4}{*}{$2^{-5}$}
								&$2^{-4}$ &7200     &8   &\cred{0.021}    &342  &0.61    &5879   &1.52  \\
								&$2^{-5}$ &30752    &8   &\cred{0.045}    &1275 &6.52    &24428  &47.58  \\
								&$2^{-6}$ &127008   &8   &\cred{0.11}     &5002 &63.61   &88205  &563.18 \\
								&$2^{-7}$ &516128   &8   &\cred{0.44}     &20072&904.18  &$>$200000&-\\
								\hline
								\multirow{4}{*}{$2^{-6}$}
								&$2^{-4}$ &14400    &8   &\cred{0.031}    &122  &0.41    &9428   &8.09  \\
								&$2^{-5}$ &61504    &8   &\cred{0.074}    &434  &3.48    &38225  &195.13\\
								&$2^{-6}$ &254016   &8   &\cred{0.23}     &1709 &38.11   &154255 &1416.21 \\
								&$2^{-7}$ &1032256  &9   &\cred{0.99}     &6907 &661.72  &$>$200000&-\\
								\hline
								\multirow{4}{*}{$2^{-7}$}
								&$2^{-4}$ &28800    &10   &\cred{0.055}   &134  &0.67    &10752   &21.37  \\
								&$2^{-5}$ &123008   &10   &\cred{0.15}    &186  &2.37    &45062   &282.98  \\
								&$2^{-6}$ &508032   &10   &\cred{0.55}    &270  &11.94   &182323  &3262.76  \\
								&$2^{-7}$ &2064512  &10   &\cred{3.06}    &976  &264.01  &$>$200000 &- \\
								\hline
						\end{tabular}}
					\end{center}
				}
			\end{table}	
		}
	\end{example}

	\section{Conclusions}\label{sec:conclusions}
		We have proposed a novel PinT preconditioner $\mathcal{P}_{\alpha}$ for the symmetrized all-at-once system of wave equations, $\mathcal{Y} \mathcal{T} \mathbf{u} = \mathcal{Y}\mathbf{f}$, whose effectiveness is supported by both theory and numerical evidence. Our preconditioner does not only extend the scope of application of the block $\alpha$ circulant preconditioning technique to MINRES solver, but also significantly improves the performance of the ABC preconditioner proposed in \cite{doi:10.1137/16M1062016}. As shown in the numerical tests, our proposed preconditioner improves significantly the original ABC preconditioner proposed in \cite{doi:10.1137/16M1062016} in terms of both CPU time and iteration number. \cred{The numerical evidence also substantiates the advantages of our proposed strategy, particularly in scenarios where the spatial grid is uniformly partitioned.}

	\section*{Acknowledgments}
	The work of Sean Hon was supported in part by the Hong Kong RGC under grant 22300921, a start-up {grant} from the Croucher Foundation, and a Tier 2 Start-up Grant from Hong Kong Baptist University. The work of Xue-lei Lin was supported by by research grants: 2021M702281 from China Postdoctoral Science Foundation; HA45001143, HA11409084 two start-up Grants from Harbin Institute of Technology, Shenzhen.

	\bibliographystyle{plain}

	\begin{appendices}
	\cred{
	\section{Proof of Lemma \ref{calphanonsinglm}}
				\begin{proof}
					When $\alpha=0$, $\mathcal{C}_{\alpha}=\mathcal{T}$. In such case, $\sigma(\mathcal{C}_{\alpha})=\sigma(\mathcal{T})=\sigma(L)\subset(1,+\infty)\subset\mathbb{C}\setminus(-\infty,0]$. Hence, $\mathcal{C}_{0}\in\mathcal{Q}(mn)$.
					
					It remains to focus on the case $\alpha\in(0,1)$. We recall from \eqref{calphaprediagform} that $$\sigma(\mathcal{C}_{\alpha})=\sigma(\Lambda_{\alpha,1}\otimes L+\Lambda_{\alpha,2}\otimes I_m).$$ Hence, for any $\mu\in\sigma(\mathcal{C}_{\alpha})$ ($\alpha\in(0,1)$), there exists $k\in\{1,2,...,n\}$ and $\lambda\in\sigma(L)\subset(1,+\infty)$ such that
					\begin{align*}
					\mu&=\lambda \lambda_{1,k}^{(\alpha)}+\lambda_{2,k}^{(\alpha)}\\
					&=\lambda-2\alpha^{\frac{1}{n}}\theta_n^{(k-1)}+\lambda\alpha^{\frac{2}{n}}\theta_n^{2(k-1)}\\
					&=\lambda-2\alpha^{\frac{1}{n}}\cos(\omega_k)+\lambda\alpha^{\frac{2}{n}}\cos(2\omega_k)-2\alpha^{\frac{1}{n}}{\bf i}\sin(\omega_k)+\lambda\alpha^{\frac{2}{n}}{\bf i}\sin(2\omega_k)\\
					&=\lambda-2\alpha^{\frac{1}{n}}\cos(\omega_k)+\lambda\alpha^{\frac{2}{n}}(2\cos^2(\omega_k)-1)+{\bf i}2\alpha^{\frac{1}{n}}\sin(\omega_k)[\lambda\alpha^{\frac{1}{n}}\cos(\omega_k)-1]\\
					&=\lambda(1-\alpha^{\frac{2}{n}})+2\alpha^{\frac{1}{n}}\cos(\omega_k)[\lambda\alpha^{\frac{1}{n}}\cos(\omega_k)-1]+{\bf i}2\alpha^{\frac{1}{n}}\sin(\omega_k)[\lambda\alpha^{\frac{1}{n}}\cos(\omega_k)-1],
					\end{align*}
					where $\omega_k=\frac{2\pi(k-1)}{n}$.
					We now show $\mu\in\mathbb{C}\setminus(-\infty,0]$ case by case.
					
					If $\cos(\omega_k)< 0$, then
					\begin{equation*}
						\Re(\mu)=   \lambda(1-\alpha^{\frac{2}{n}})+2\alpha^{\frac{1}{n}}\cos(\omega_k)[\lambda\alpha^{\frac{1}{n}}\cos(\omega_k)-1]>\lambda(1-\alpha^{\frac{2}{n}}) > 0,
					\end{equation*}
					which implies $\mu\in\mathbb{C}\setminus(-\infty,0]$.
					
					If $\cos(\omega_k)=0$, then
					\begin{equation*}
						\Re(\mu)=\lambda(1-\alpha^{\frac{2}{n}})> 0,
					\end{equation*}
					which  implies $\mu\in\mathbb{C}\setminus(-\infty,0]$.
					
					If $\frac{1}{\lambda \alpha^{\frac{1}{n}}}>\cos(\omega_k)>0$ and $\cos(\omega_k)\neq 1$, then $\Im(\mu)\neq 0$.  If $\cos(\omega_k)=1$, then $\Re(\mu)=\lambda(1-\alpha^{\frac{2}{n}})+2\alpha^{\frac{1}{n}}(\lambda\alpha^{\frac{1}{n}}-1 )=(\alpha^{\frac{1}{n}}-1)^2+(\lambda-1)(\alpha^{\frac{2}{n}}+1)>0$. That means $\mu\in\mathbb{C}\setminus(-\infty,0]$ when $\frac{1}{\lambda \alpha^{\frac{1}{n}}}>\cos(\omega_k)>0$.

						If $\cos(\omega_k)=\frac{1}{\lambda \alpha^{\frac{1}{n}}}$, then $\mu=\lambda(1-\alpha^{\frac{2}{n}})>0$, which  implies $\mu\in\mathbb{C}\setminus(-\infty,0]$.
					
					If  $\cos(\omega_k)>\frac{1}{\lambda \alpha^{\frac{1}{n}}}$, then $\Re(\mu)=\lambda(1-\alpha^{\frac{2}{n}})+2\alpha^{\frac{1}{n}}\cos(\omega_k)[\lambda\alpha^{\frac{1}{n}}\cos(\omega_k)-1]>0$,  which  implies $\mu\in\mathbb{C}\setminus(-\infty,0]$.
					
					Hence, if $\alpha\in(0,1)$, then for any $\mu\in\sigma(\mathcal{C}_{\alpha})$ it holds $\mu\in\mathbb{C}\setminus(-\infty,0]$. In other words, when $\alpha\in(0,1)$, it holds that $\mathcal{C}_{\alpha}\in \mathcal{Q}(mn)$.
					
					To summarize, for any $\alpha\in[0,1)$, it holds that $\mathcal{C}_{\alpha}\in \mathcal{Q}(mn)$. The proof is complete.
				\end{proof}
				}

			\end{appendices}

		\end{document}